\documentclass[11pt,oneside]{amsart}

\usepackage{amssymb,amsmath,enumitem,array}
\usepackage{tikz-cd}
\usepackage{xcolor} 
\colorlet{mdtRed}{red!50!black}
\definecolor{dblue}{rgb}{0,0,.6}
\usepackage[colorlinks]{hyperref}
\hypersetup{linkcolor=blue,citecolor=dblue,filecolor=dullmagenta,urlcolor=mdtRed}

\usepackage{multirow} 
\usepackage{booktabs}
\usepackage{geometry}
\newtheorem{theorem}[equation]{Theorem}
\newtheorem{lemma}[equation]{Lemma}
\newtheorem{proposition}[equation]{Proposition}
\newtheorem{corollary}[equation]{Corollary}

\theoremstyle{definition}
\newtheorem{definition}[equation]{Definition}

\newtheorem{example}[equation]{Example}

\usepackage[nosort,nocompress]{cite}

\numberwithin{equation}{section}

\newenvironment{remark}
 { \refstepcounter{equation}
   \medskip\noindent
   \textbf{Remark \thesection.\arabic{equation}.} }
 { \hfill \qedsymbol}

\DeclareMathOperator{\rk}{Rk}

\DeclareMathOperator{\HF}{HF}

\renewcommand{\geq}{\geqslant}
\renewcommand{\leq}{\leqslant}
\everymath{\displaystyle}

\title{On the variety of sum of powers of certain complete intersection forms}

\author[Bhat]{Meghana Bhat}
\address[Meghana Bhat]{College of Computing and Data Science, Nanyang Technological University, 50 Nanyang Ave, Block N 4, Singapore - 639798}
\email{meghana.bhat@ntu.edu.sg}

\author[Sebastian]{Ronnie Sebastian}
\address[Ronnie Sebastian]{Department of Mathematics, Indian Institute of Technology Bombay, Powai, Mumbai - 400076, Maharashtra, India}
\email{ronnie@iitb.ac.in}

\thanks{MB is supported by the SUG Grant WBS 03INS002590C130 funded by NTU Singapore.}

\date{\today}
\begin{document}

\maketitle
\begin{abstract}
Let $F$ be a homogeneous polynomial (form). 
If the annihilator ideal $F^\perp$ is a complete 
intersection, then $F$ is called as a 
complete intersection form. For many complete 
intersection forms, we show that its minimal 
apolar ideal is a complete intersection and 
further we show that ${\rm VSP}(F)$
is irreducible and compute its dimension. 

\end{abstract}

\section{Introduction}
    Let $S=\mathbb{C}[x_0, \ldots, x_n]$ be 
    the standard graded polynomial ring, and 
    let  $F \in S$ be a homogeneous polynomial 
    of degree $d$, also known as a degree $d$ 
    form. It is well known that there exist linear forms
    $L_1, \ldots, L_r \in S_1$ such that 
    \[F=\sum_{i=1}^{r}\alpha_iL_i^d,\]  for 
    some $\alpha_i \in \mathbb{C}$, $1 \leqslant i \leqslant r$. 
    Such a decomposition of $F$ is called a 
    {\it Waring decomposition} or a {\it sum of powers decomposition}.  
    The least integer $r$ such that there exists 
    exactly $r$ summands is  called the {\it Waring rank} 
    of $F$, denoted by $\rk(F)$. 

    We now briefly mention some earlier 
    works related to the results in this article. 
    Given a homogeneous form $F$, computing 
    the Waring rank of $F$ is a well-known problem. 
    J. Alexander and A. Hirschowitz computed the rank of 
    a generic form \cite{AH95}. In view of this result, interest 
    shifted to studying the Waring rank of specific families. 
    A lower bound on the Waring rank of monomials was given in 
    \cite{RS11}. {The precise value of the rank for monomials was 
    computed in \cite{CCG12, BBT13}. Moreover, in 
    \cite{CCG12} the authors solve the Waring 
    problem for a form $F$ which is a sum of pairwise 
    coprime monomials.}  
    Other cases for which the Waring rank has been studied are
    elementary symmetric polynomials in \cite{Lee16}, binary 
    binomials in \cite{BM21}, some reducible forms 
    \cite{BCDM26,CCCGW18, CVG16,CG26}.
    Approaching the computation of the Waring rank of 
forms with the sole information of annihilator ideals is a 
relatively new and different approach. 
In \cite{BCMS26}, authors define a 
\emph{complete intersection form} 
to be a homogeneous polynomial whose annhilator ideal 
(see Section \ref{Sec:Background} for a definition) 
is a complete intersection. They compute the Waring 
rank of complete intersection forms in many cases. 
    Despite the understanding of how to compute the Waring 
    rank in many cases, in general, given a form, computing 
    the Waring rank remains difficult. 
    There are also other notions of rank, cactus
    rank, smoothable rank, border rank, and we refer the interested 
    reader to \cite{RS11} for their definitions and comparisons 
    between them. See also \cite{BB15}, \cite{BD20}. 


A linear form $\ell$ can be represented as a point in the projective
    space $\mathbb P^n$. The variety of sum of powers 
    (see \cite{RS00}) of a given form $F$, with respect to some 
    non-negative integer $r$, is defined as
\[
{\rm VSP}(F,r)=\left\{\{[\ell_1],[\ell_2],\ldots,[\ell_r]\}\in {\rm Hilb}_r(\mathbb{P}^n) \,|\, F=\ell_1^d+\ell_2^d+\ldots+\ell_r^d\right\}.
\]
where ${\rm Hilb}_r(\mathbb{P}^n)$ denotes the 
Hilbert scheme of length $r$ zero-dimensional subschemes 
of $\mathbb{P}^n$. These varieties are of interest, for example see 
\cite{Mukai}, \cite{RS00}, \cite{AR-1}, \cite{RV}, \cite{MM-Birat}. 
It is interesting to understand these varieties when $r$
is the Waring rank of $F$. For ease of notation, we denote 
${\rm VSP}(F,\rk(F))$ by ${\rm VSP}(F)$.
Let $\mathbb{X} \subset \mathbb{P}^n$ be a finite set 
of closed points. We say $\mathbb{X}$ is an \emph{apolar set} to $F$ if 
its reduced ideal satisfies  
$I(\mathbb{X}) \subset F^{\perp}$. Further, $\mathbb{X}$ 
is said to be a \emph{minimal apolar set} if 
$|\mathbb{X}|= \rk(F)$.
In the case when $F$ is a monomial, in \cite{BBT13} the authors 
prove that ${\rm VSP}(F)$ is irreducible, compute its  
dimension and show that every minimal apolar ideal
is a complete intersection. In \cite{BCDM26}, 
the dimension of VSP of certain 
reducible quartic forms and the generators of 
minimal apolar ideals were described.


The study of the Hilbert function of a minimal apolar 
ideal of the given polynomial provides very useful 
information on a lower bound of the Waring rank of 
$F$ (see \cite{CCCGW18,CCG12, BCDM26}). 
In \cite[Question 2.11]{MO20} authors asked that 
for a given $F \in S_d$, if $\mathbb{X}$, 
$\mathbb{X}'$ are minimal apolar sets of $F$, 
then is it true that the Hilbert function of 
$I(\mathbb{X})$ and $I(\mathbb{X}')$ are the same? 
In general, it is not true that every minimal apolar 
set of a given form has the same Hilbert function 
(see \cite{ACO24} for an example that negatively 
answers this question and explores connections of 
the Hilbert function of points with \emph{liaison theory}). 
But there are many forms for which the answer is 
positive. For example: binary forms 
(by Sylvester's algorithm), monomials, etc.  

We now state the main results of this article, which 
are inspired from \cite{BBT13}. Let $T=\mathbb C[X_0,X_1,\ldots,X_n]$.
We say a form $f$ is non-reduced if there are forms $h,g$
with ${\rm deg}(h)\geqslant 1$ such that $f=h^2g$.
\begin{theorem}\label{theorem-intro1}
Let $F^\perp = (f_0, f_1, \ldots,f_n) \subseteq T$ be a complete intersection.
Let $\deg f_i = a_i$ and assume that  $a_0 \leqslant a_1\leqslant \ldots \leqslant a_n$. 
Assume that $f_i$ is non-reduced for some $i$. Let $k$ be the smallest 
such that $f_k$ is non-reduced. Assume that the subscheme 
$V(f_0,f_1,\ldots,f_{k-1})\subset \mathbb P^n$
is reduced.  
Let $J$ be the complete-intersection $(f_0,\ldots,f_{k-1},f_{k+1},\ldots,f_n)$
and let $HF(T/J,-)$ be the Hilbert function of $T/J$. Then ${\rm VSP}(F)$
is irreducible of dimension 
$$\sum_{\substack{i=0\\i\neq k}}^n HF(T/J,a_i-a_k)\,.$$
Moreover, if $\mathbb X\subset \mathbb P^n$ 
is a minimal apolar set to $F$, then
the ideal $I(\mathbb X)$ is a complete intersection of the form 
$$\big(f_i - \psi_if_k \;\big|\; i \in \{0,\ldots,n\}\setminus\{k\}\big)\,,$$
where $\psi_i$ is homogeneous of degree $a_i-a_k$.
\end{theorem}

\begin{theorem}\label{theorem-intro2}
Let $F^\perp=(f_0,f_1,\ldots,f_n)$ be a complete intersection. 
Let $a_i={\rm deg}(f_i)$ and assume that $a_0\leqslant a_1\leqslant \ldots\leqslant a_n$. 
Further assume that $a_n\geqslant \left(\sum\nolimits_{i=0}^{n-1}a_i\right)-n+2$
and that the subscheme $V(f_0,f_1,\ldots,f_{n-1})\subset \mathbb P^n$
is reduced. 
Let $J$ be the complete-intersection $(f_0,\ldots,f_{n-1})$
and let $HF(T/J,-)$ be the Hilbert function of $T/J$. Then ${\rm VSP}(F)$
is irreducible of dimension 
$$\sum_{\substack{i=0}}^{n-1} HF(T/J,a_i-a_n)\,.$$    
Moreover, if $\mathbb X\subset \mathbb P^n$ 
is a minimal apolar set to $F$, then
the ideal $I(\mathbb X)$ is a complete intersection of the form 
$$\big(f_i - \psi_if_n \;\big|\; i \in \{0,\ldots,n-1\}\big)\,,$$
where $\psi_i$ is homogeneous of degree $a_i-a_n$. 
\end{theorem}

In Example \ref{counterexample} we show that the assumption 
$V(f_0,f_1,\ldots,f_{k-1})\subset \mathbb P^n$
is reduced is necessary in Theorem \ref{theorem-intro1}.

Our results are obtained by using 
a combination of methods from \cite{CCCGW18} and 
\cite{BBT13}. 
It is interesting to note that many of the complete intersection 
forms share common behaviour with monomials 
(in terms of Waring rank, VSP, minimal apolar ideals).

\noindent {\bf Acknowledgements}:  
{MB 
is grateful to Ananthnarayan H., Enrico Carlini and Shreedevi Masuti for 
their encouragement. She thanks RS and 
IIT Bombay for hospitality, where the project was 
initiated during her visit.}

\section{Preliminaries}\label{Sec:Background}
A useful tool for studying Waring rank problems is 
the apolarity action. Recall the notation 
$T=\mathbb{C}[X_0, \ldots, X_n]$ is the polynomial 
ring over $\mathbb{C}$. The apolarity action of $T$ on 
$S=\mathbb C[x_0,\ldots,x_n]$ is defined on monomials by
\[X_0^{a_0} \cdots X_n^{a_n} \circ (x_0^{b_0}\cdots x_n^{b_n}) = 
\frac{\partial^{a_0}}{\partial x_0^{a_0}}\cdots \frac{\partial^{a_n}}{\partial x_n^{a_n}}(x_0^{b_0}\cdots x_n^{b_n}),\]
and then extended linearly on $S$. 
Given a homogeneous form $F \in S_d$, using the 
apolarity action, we write annihilator ideal of $F$ as 
$$F^{\perp} = \{ g \in T \mid g \circ F = 0\}\,.$$ 
Note that $F^{\perp}$ is an Artinian Gorenstein ideal. 
Given a Waring decomposition 
$$F = \sum_{i=1}^{r}\alpha_iL_i^d\,,$$ 
with $L_i = \sum\nolimits_{j=0}^{n}a_{ij}x_j$, we associate points 
$P_i \in \mathbb{P}^n$ with $L_i$ for all 
$1 \leqslant i \leqslant r$ given by 
$P_i = [a_{i0}: \cdots:a_{in}]$. The following Lemma 
describes the connection between apolarity and Waring problems.  
We refer the reader to \cite[Lemma 1.15]{IK99} for a proof.
\begin{lemma}[Apolarity Lemma]\label{Lemma:Apolarity}
Let $\mathbb{X}=\{P_1, \ldots, P_r\} \subset \mathbb{P}^n$ be 
a set of points. Then for some $\alpha_i \in \mathbb{C}$, we 
have $F = \sum\nolimits_{i=1}^r\alpha_i L_i^d$ if and only if 
$I(\mathbb{X}) \subset F^{\perp}$, where $L_i$ is the linear 
form corresponding to the point $P_i$ for $1 \leqslant i \leqslant r$.
\end{lemma}
    
Let {$\mathbb{X} \subset \mathbb{P}^n$} be a set of points. 
Then $\mathbb{X}$ is  said to be apolar to $F$ if 
$I(\mathbb{X}) \subset F^{\perp}$, and we call $I(\mathbb{X})$ 
an {\it apolar ideal} of $F$.  Further, $\mathbb{X}$ is a {\it minimal apolar 
set} of $F$ if $|\mathbb{X}| = \rk(F)$.

We recall the following result from \cite{BCMS26}.  

\begin{theorem}[\citenum{BCMS26}, Theorem 3.4]\label{thm2}
    Let $F^\perp = (f_0, f_1, \ldots,f_n) \subseteq T$ be a complete intersection.
    Let $\deg f_i = a_i$ and assume that  $a_0 \leqslant a_1\leqslant \ldots \leqslant a_n$. 
    Assume that $f_i$ is non-reduced for some $i$. Let $k$ be the smallest 
    such that $f_k$ is not reduced and assume that the subscheme 
    $V(f_0,f_1,\ldots,f_{k-1})\subset \mathbb P^n$
    is reduced. Then $\rk(F)=\frac{1}{a_k}\prod\nolimits_{i=0}^n a_i$.
\end{theorem}

\section{Computation of VSP}
Consider $F^\perp = (f_0, \dots, f_{n})$ a complete intersection ideal and assume that 
$f_k=h^2g$ for some $h$ with ${\rm deg}(h)\geqslant 1$.
Let $J$ be the ideal 
$$J:=(f_0,\ldots,f_{k-1},f_{k+1},\ldots,f_n)\,.$$
Thus, 
$$F^\perp = J + (h^2 g)\,.$$

Let $\mathbb X \subset \mathbb{P}^n$ be a reduced subscheme of dimension zero. 
Let $\mathcal{I} := I(\mathbb X)$ be its 
ideal consisting of forms vanishing on $\mathbb X$. 
In particular, $\mathcal I$ is a radical ideal. 
Assume that $\mathcal{I}\subset F^\perp$.
Our goal is to compare the Hilbert 
function of $\mathcal{I}$ to that of the complete intersection $J$.

By assumption, $J$ is a complete intersection of 
codimension $n$, so $T/J$ is a complete intersection of dimension 1. 
Since $f_0, \dots, f_{n}$ forms a regular sequence,
$f_k=h^2 g$ is not a zero-divisor modulo $J$.
Thus, $h$, $h^2$, $g$, and $hg$ are not zero-divisors on 
$T/J$. 

The following result shows that in our setup, 
minimal apolar ideals have the same Hilbert function.

\begin{proposition}\label{prop:hilbert_eq}
Let $F^\perp = (f_0, f_1, \ldots,f_n) \subseteq T$ be a complete intersection.
Let $\deg f_i = a_i$ and assume that  $a_0 \leqslant a_1\leqslant \ldots \leqslant a_n$. 
    Assume that $f_i$ is non-reduced for some $i$. Let $k$ be the smallest 
    such that $f_k$ is not reduced. Let $f_k=h^2g$ 
    for some $h$ with ${\rm deg}(h)\geqslant 1$. Assume that the subscheme 
    $V(f_0,f_1,\ldots,f_{k-1})\subset \mathbb P^n$
    is reduced.  
Let $J$ be the ideal $(f_0,\ldots,f_{k-1},f_{k+1},\ldots,f_n)$ and $\mathbb X \subset \mathbb{P}^n$ be a finite set such that $|\mathbb X|=\rk(F)$.
Then for all $i \in \mathbb{Z}$, we have 
\[
\HF(T/I(\mathbb{X}), i) = \HF(T/J, i).
\]
\end{proposition}
\begin{proof}
Recall the definition of $e$-computability from \cite[Definition 3.5]{CCCGW18}.
In Theorem \ref{thm2}, 
the rank of $F$ is computed by $h$, as can be see from 
equation (3.1) in the proof of \cite[Proposition 3.2]{BCMS26}. 
Applying \cite[Proposition 3.6]{CCCGW18}
we get
\begin{align}
    I(\mathbb{X}):(h) & = I(\mathbb{X}) \label{EQ:I:h} \\
    I(\mathbb{X}) +(h) & =F^\perp +(h)= J+(h). \label{EQ:I+h=J+h}
\end{align}
By Equation \eqref{EQ:I:h} we have $h$ is a nonzerodivisor on $I(\mathbb{X}).$
Hence we have 
\begin{equation}\label{EQ:HS1}
  (1-t^{\deg h}) H(T/I(\mathbb{X}), t)  = H(T/(I(\mathbb{X})+(h)), t),
\end{equation}
where $H(T/I(\mathbb{X}), t)$ denotes the Hilbert series of $T/I(\mathbb{X})$ in the variable $t.$
Similarly, $h$ is also a nonzerodivisor in $T/J.$ Therefore
\begin{equation}\label{EQ:HS2}
  (1-t^{\deg h}) H(T/J, t)  = H(T/(J+(h)), t).
\end{equation}
By Equation \eqref{EQ:I+h=J+h}, RHS of Equations \eqref{EQ:HS1} 
and \eqref{EQ:HS2} are equal. Hence we obtain
$$\HF(T/I(\mathbb{X}), i) = \HF(T/J, i)$$ 
for $i \geq 0.$
This completes the proof of the Proposition. 
\end{proof}

The following lemma is a technical tool 
to obtain a generating set for minimal 
apolar ideals of forms considered in our setup.
\begin{lemma}\label{Lem:F^{perp}=I+h^2g}
    Let $F^\perp = (f_0,f_1,\ldots,f_n) \subseteq T$ be a complete intersection.
Let $\deg f_i = a_i$ and assume that $a_0 \leqslant a_1 \leqslant \cdots \leqslant
a_n$. Assume that $f_i$ is non-reduced for some $i$. Let $k$ be the smallest
such that $f_k$ is not reduced. 
Assume that the subscheme 
    $V(f_0,f_1,\ldots,f_{k-1})\subset \mathbb P^n$
    is reduced. 
Let $f_k = h^2g$ where $\deg(h)\geqslant 1$. If $I(\mathbb{X}) \subseteq F^{\perp}$ is any minimal apolar ideal, then 
\[
I(\mathbb{X}) +(f_k) = F^{\perp}.
\]
\end{lemma}
\begin{proof} 
Let $J$ be the ideal $(f_0,\ldots,f_{k-1},f_{k+1},\ldots,f_n)$.
We first claim that 
\begin{equation}\label{EQ:I+h^2=J+h^2}
    I(\mathbb{X}) +(h^2) = J+(h^2).
\end{equation}
Note that we have $I(\mathbb{X}) \subset F^{\perp} \subset J+(h^2).$ 
Hence, $I(\mathbb{X}) +(h^2) \subseteq J+(h^2)$. 
By \eqref{EQ:I:h}, $h$ is a nonzerodivisor on $T/I(\mathbb X)$
and so the same is true for $h^2$. Clearly, $h^2$ is a 
nonzerodivisor on $T/J$. This gives
\begin{align*}
\HF(T/I(\mathbb{X}), i)- \HF(T/I(\mathbb{X}), i-2\deg(h))&=
    \HF(T/(I(\mathbb{X})+(h^2)), i)\\
    &\geq \HF(T/(J+(h^2)), i)\\
    &=\HF(T/J, i)- \HF(T/J, i-2\deg(h))\,.
\end{align*}
By Proposition \ref{prop:hilbert_eq}, the extremes are equal
and so we have equality everywhere in the above. This
shows that $I(\mathbb{X}) +(h^2) = J+(h^2)$.

Next we claim that $(F^\perp:h^2)=J+(g)$. If $h^2x=j+h^2gp$,
where $j\in J$, then $h^2(x-gp)\in J$, which shows that 
$x-gp\in J$. It follows that $(F^\perp:h^2)\subset J+(g)$. 
The other inclusion is obvious. This proves the claim. 

Since $F^{\perp} \subset J+(h^2)$, by 
equation \eqref{EQ:I+h^2=J+h^2} we obtain 
$F^{\perp}\subseteq I(\mathbb X)+(h^2)$. Hence, for 
$u \in F^{\perp}$, there exist 
$v \in I(\mathbb{X})$ and $h' \in T$ such that
\[
u = v +h^2h'.
\]
Note that $u-v \in F^{\perp}$. Therefore 
\begin{equation}\label{EQ:h^2colon}
    h^2h' \in (F^{\perp} \cap (h^2)) = h^2(F^{\perp}\colon (h^2)).
\end{equation}
Moreover, $F^{\perp}:(h^2) = J+(g).$ Substituting this in Equation \eqref{EQ:h^2colon} we obtain
\begin{equation}
    h^2h' \in h^2(J+(g)) = (h^2)J+(h^2g)=(h^2)J+(f_k).
\end{equation}
Hence we obtain $u \in I(\mathbb{X}) +(h^2)J+(f_k).$
Therefore, $F^{\perp} \subseteq I(\mathbb{X}) +(h^2)J+(f_k)$.
Since the reverse inclusion is trivial, we obtain the
equality
\begin{equation}\label{EQ:NAK}
    F^{\perp} = I(\mathbb{X})+(h^2)J +(f_k).
\end{equation}
Suppose $I(\mathbb{X}) +(f_k)\subsetneqq F^{\perp}$. 
Choose a homogeneous polynomial 
$H\in F^{\perp}\setminus (I(\mathbb{X}) +(f_k))$ of 
least degree possible. Using \eqref{EQ:NAK}, write 
$H=v+ h^2v'$ where $v\in I(\mathbb X)+(f_k)$ and $v' \in J$. 
Since $\deg v' <\deg H$, $v' \in I(\mathbb X)+(f_k).$ 
This means that $H=v+ h^2v'\in I(\mathbb X)+(f_k)$, 
a contradiction. This completes the proof of the Lemma.
\end{proof}

We now describe the generators of any minimal apolar ideal.
\begin{theorem}\label{vspthm-unified}
    Let $F^\perp = (f_0,f_1,\ldots,f_n) \subseteq T$ be a complete intersection.
Let $\deg f_i = a_i$ and assume that $a_0 \leqslant a_1 \leqslant \cdots \leqslant
a_n$. Assume that $f_i$ is non-reduced for some $i$. Let $k$ be the smallest
such that $f_k$ is not reduced. Let $f_k = h^2g$ where $\deg(h)\geqslant 1$.
Assume that the subscheme $V(f_0,\ldots,f_{k-1}) \subset \mathbb P^n$ is reduced.
If $ I(\mathbb X) \subset F^\perp$ is any minimal apolar ideal, 
then $I(\mathbb{X})$ is a complete intersection, and there exist $\phi_i \in T_{a_i-a_k}$ such that
\[
 I(\mathbb{X}) = \big(f_i - \phi_if_k \;\big|\; i \in \{0,\ldots,n\}\setminus\{k\}\big).
\]
\end{theorem}
\begin{proof}
    By Lemma \ref{Lem:F^{perp}=I+h^2g}, we have $F^{\perp} = I(\mathbb{X})+(f_k)$ for any minimal apolar ideal $I(\mathbb{X})$ to $F$. 
    Therefore, for each $0 \leq i \leq n$, and $i\neq k$, we have
    \begin{equation}
        f_i = v_i + \phi_{i}f_k
    \end{equation}
    for some $\phi_i \in T_{a_i-a_k}$. Let 
    $I :=\big(f_i - \phi_if_k \;\big|\; i \in \{0,\ldots,n\}\setminus\{k\}\big)$. Note that 
    $I+(f_k) = F^{\perp}.$ Hence, $I$ is generated by a 
    regular sequence. Since $I \subset I(\mathbb{X})$ 
    and both have the same Hilbert function, it follows
    that we have equality.
\end{proof}

\begin{remark}
    In Theorem \ref{vspthm-unified} above, the hypothesis that $V(f_0,f_1,\ldots,f_{k-1})\subset \mathbb P^n$ is a reduced subscheme
    is necessary. See Example \ref{counterexample}.
\end{remark}

{We now compute the dimension of the 
VSP of forms considered in the previous section.}
Let ${\rm Hilb}_{\rk(F)}(\mathbb P^n)$ denote the Hilbert
scheme of zero dimensional subschemes of $\mathbb P^n$ 
of length $\rk(F)$. 
\begin{definition}\label{def mathcal V}
Let $\mathcal V$ denote the subset of 
${\rm Hilb}_{\rk(F)}(\mathbb P^n)$ consisting of subschemes 
$\mathbb X$ which are reduced and such that $I(\mathbb X)\subset F^\perp$.     
\end{definition}
\noindent 
Recall that 
${\rm VSP}(F)\subset {\rm Hilb}_{\rk(F)}(\mathbb P^n)$ is defined
to be the closure of $\mathcal V$. Thus, for example, 
with the hypothesis of Theorem \ref{vspthm-unified},
if we also assume that $k=n$ and $a_{n-1}<a_n$, then 
we conclude that ${\rm VSP}(F)$ consists of a single 
point. 

Now let the hypothesis be as in Theorem \ref{vspthm-unified}. 
Consider the family 
$$\mathcal B:=\prod_{\substack{i=0\\i\neq k}}^nT_{a_i-a_k}\,.$$ 
We argue as in the proof 
of \cite[Proposition 23]{BBT13}. 
Consider the open subset $U\subset \mathcal B$ 
consisting of elements $(\phi_{0},\ldots, \phi_{k-1},\phi_{k+1},\ldots,\phi_n)$
such that the ideal 
$\mathcal I=(f_0-\phi_{0}f_k,\ldots,f_{k-1}-\phi_{k-1}f_k,f_{k+1}-\phi_{k+1}f_k,\ldots, f_n-\phi_nf_k)$
defines a reduced subscheme of $\mathbb P^n$. 
It is clear that $U$ maps surjectively onto $\mathcal V$. 
It follows that with the hypothesis as in 
Theorem \ref{vspthm-unified}, ${\rm VSP}(F)$ is irreducible. 

Finally we want to compute the dimension of ${\rm VSP}(F)$.
{ We do this by finding a variety 
$\mathcal B'$ along with a morphism of schemes 
$\mathcal B'\to  {\rm Hilb}_{\rk(F)}(\mathbb P^n)$
which is bijective onto $\mathcal V$. 
As a result we get that 
$\dim (\mathcal B')=\dim({\rm VSP}(F))$.}
As before, $(f_0,\ldots,f_n) \subseteq T$ 
is a complete intersection with
$\deg f_i = a_i$ and 
$a_0 \leqslant \cdots \leqslant a_n$. Fix $k \in \{0,\ldots,n\}$
and let $J = (f_0,\ldots,f_{k-1},f_{k+1},\ldots,f_n)$. 
We first inductively construct 
subspace $C_d\subset T_d$
such that $T_d=J_d\oplus C_d$. Define $C_0=T_0=\mathbb C$. 
Assume we have defined $C_i$ for $i\leqslant d-1$ such that  
\begin{enumerate}
    \item[(a)] $T_i=J_i\oplus C_i$,
    \item[(b)] $f_kC_{i-a_k}\subset C_i$. 
\end{enumerate} 
We claim that $J_d\cap f_kC_{d-a_k}=0$. 

If $f_kx\in J_d$
then $x\in J_{d-a_k}\cap C_{d-a_k}=0$. Choose $C_d$
to be any complement of $J_d$ containing $f_kC_{d-a_k}$.
This shows we may find $C_i$'s satisfying the above 
two conditions. 

\begin{lemma}\label{prelim lem uniqueness}   
Let 
$\mathcal I=(f_0-\phi_{0}f_k,\ldots,f_{k-1}-\phi_{k-1}f_k,f_{k+1}-\phi_{k+1}f_k,\ldots, f_n-\phi_nf_k)$,
where $\phi_{i}\in T_{a_i-a_k}$. 
If $\beta\in \mathcal I_d\cap C_d$ then $\beta=0$.
\end{lemma}
\begin{proof}
Let $\beta_0:=\beta$. We will find a sequence of 
elements $\beta_i$, for $i\geqslant 1$, such that 
$\beta_{i-1}=(-1)^i\beta_if_k$. Then it will follow
that $\beta=0$. {Since 
$\beta \in \mathcal I$, there exist $\delta_i$, for 
$i\in\{0,1,\ldots, n\}\setminus \{k\}$, such that}
    \begin{align*}
        \beta=\beta_0&=\sum_{i=0}^{k-1}\delta_i(f_{i}-\phi_if_k)+\sum_{i=k+1}^n\delta_i(f_{i}-\phi_if_k)\\
            &=\sum_{i=0}^{k-1}\delta_if_i+\sum_{i=k+1}^n\delta_if_{i} -\left(\sum_{\substack{i=0\\i\neq k}}^n\delta_i\phi_i\right)f_k\\
            &=\left(\sum_{i=0}^{k-1}\delta_if_i+\sum_{i=k+1}^n\delta_if_{i} -\alpha f_k\right)-\beta_1f_k\,,
    \end{align*}
    where 
    $$\sum_{\substack{i=0\\i\neq k}}^n\delta_i\phi_i=\alpha+\beta_1$$ 
    is the decomposition 
    in $J_{d-a_k}\oplus C_{d-a_k}$. Since $\beta_1\in C_{d-a_k}$ 
    we obtain $\beta_1f_k\in C_d$. 
    Clearly, $J_{d-a_k}f_k\subset J_d$ and so 
    $$\left(\sum_{i=0}^{k-1}\delta_if_i+\sum_{i=k+1}^n\delta_if_{i} -\alpha f_k\right)\in J_d.$$ 
    Hence, $\beta_0+\beta_1f_k\in C_d\cap J_d=0$. 
    Thus, $\beta_0=-\beta_1f_k$.
    As $(\mathcal I:f_k)=\mathcal I$, we obtain
    $\beta_1\in \mathcal I_{d-a_k}\cap C_{d-a_k}$.
    Proceeding in this fashion we see that $\beta_0=-\beta_1f_k=\beta_2f_k^2=\ldots$.
    This shows that $\beta=0$.
\end{proof}

\begin{lemma}\label{unique in C_i}
    Let $\phi_i,\psi_i\in C_{a_i-a_k}$. If
    \begin{align*}
        (f_0-\phi_{0}f_k,\ldots,&f_{k-1}-\phi_{k-1}f_k,f_{k+1}-\phi_{k+1}f_k,\ldots, f_n-\phi_nf_k)=\\
        &(f_0-\psi_{0}f_k,\ldots,f_{k-1}-\psi_{k-1}f_k,f_{k+1}-\psi_{k+1}f_k,\ldots, f_n-\psi_nf_k)
    \end{align*}
    then $\phi_i=\psi_i$.
\end{lemma}
\begin{proof}
There exist $\alpha_i\in T_{a_j-a_i}$ such that 
\begin{align}\label{uniqueness e1}
    f_j-\psi_jf_k&=\sum_{\substack{i=0\\i\neq k}}^{n}\alpha_i(f_{i}-\phi_if_k)\\
        &=\sum_{\substack{i=0\\i\neq k}}^{n}\alpha_if_{i}-
            \left(\sum_{\substack{i=0\\i\neq k}}^n\alpha_i\phi_i\right)f_k\,.\nonumber
\end{align}
By rearranging the terms we obtain
$$(1-\alpha_j)f_j= \sum_{\substack{i=0\\i\notin \{k,j\}}}^{n}\alpha_if_{i}+
            \left(\psi_j-\sum_{\substack{i=0\\i\neq k}}^n\alpha_i\phi_i\right)f_k\,.$$
{ Since $f_0, \ldots, f_n$ is a regular sequence
and $1-\alpha_j\in \mathbb C$ we get $\alpha_j=1$.} 

Substitute this in  Equation
\eqref{uniqueness e1}, we get 
$$(\phi_j-\psi_j)f_k=\sum_{\substack{i=0\\i\notin \{k,j\}}}^{n}\alpha_i(f_{i}-\phi_if_k)\,.$$
By Lemma \ref{prelim lem uniqueness}
we obtain $(\phi_j-\psi_j)f_k\in \mathcal I_{a_j}\cap C_{a_j}$,
and so is zero. Hence, $\phi_j=\psi_j$.
\end{proof}

\begin{lemma}\label{lem:existence-reduced-form}
Let
\[
\mathcal I = \big(f_0-\phi_{0}f_k,\ldots,f_{k-1}-\phi_{k-1}f_k,\ f_{k+1}-\phi_{k+1}f_k,\ldots,f_n-\phi_nf_k\big)
\]
for some $\phi_i \in T_{a_i-a_k}$. Then there exist
$\psi_i \in C_{a_i-a_k}$ such that
\[
\mathcal I = \big(f_0-\psi_{0}f_k,\ldots,f_{k-1}-\psi_{k-1}f_k,\ f_{k+1}-\psi_{k+1}f_k,\ldots,f_n-\psi_nf_k\big).
\]
\end{lemma}

\begin{proof}
We first establish, by induction on $d$, that for $d$ we have the containment
\begin{equation}\label{eq:J-in-I-plus-C}
J_d \subset \mathcal I_d + f_kC_{d-a_k}.
\end{equation}

\emph{Base case $d= a_k-1$.} Since $d<a_k$, the ideal 
$J_{\leqslant d}$ is generated by those $f_i$ for which 
$a_i<a_k$. Similarly, $\mathcal I_{\leqslant d}$ is generated
by those $f_i-\phi_if_k$ for which $a_i<a_k$. But note that 
since $\phi_i\in T_{a_i-a_k}$ we get $\phi_i=0$. It follows
that $J_{\leqslant d}=\mathcal I_{\leqslant d}$.
Thus, \eqref{eq:J-in-I-plus-C} holds for every $d<a_k$.

\emph{Inductive step.} Suppose \eqref{eq:J-in-I-plus-C} 
holds in every degree strictly less than $d$, 
and let $d \geqslant a_k$. First,
\begin{equation}\label{eq:J-in-I-plus-T}
J_d = \sum_{l\neq k} f_lT_{d-a_l} \subset \mathcal I_d + f_kT_{d-a_k},
\end{equation}
since writing $f_l = (f_l - \phi_lf_k) + \phi_lf_k$ gives
\[
f_lT_{d-a_l} \subset \mathcal I_d + f_k\big(\phi_lT_{d-a_l}\big) \subset \mathcal I_d + f_kT_{d-a_k}.
\]
Now decompose $T_{d-a_k} = J_{d-a_k} \oplus C_{d-a_k}$, so that
$f_kT_{d-a_k} = f_kJ_{d-a_k} + f_kC_{d-a_k}$. By the inductive hypothesis applied
in degree $d-a_k < d$,we obtain
\[
J_{d-a_k} \subset \mathcal I_{d-a_k} + f_kC_{d-2a_k}.
\]
By using $f_kC_{d-2a_k} \subset C_{d-a_k}$, we get
\begin{equation}\label{EQ:f_kJsubset}
    f_kJ_{d-a_k} \subset f_k\mathcal I_{d-a_k} + f_k^2C_{d-2a_k} \subset \mathcal I_d + f_kC_{d-a_k}.
\end{equation}
Combining Equation \eqref{EQ:f_kJsubset} with
\eqref{eq:J-in-I-plus-T} gives
\begin{align*}
J_d &= \sum_{l\neq k} f_lT_{d-a_l} \\
&\subset \mathcal I_d + f_kT_{d-a_k}\\
&=\mathcal I_d + f_kJ_{d-a_k} + f_kC_{d-a_k}\\
&=\mathcal I_d + f_kC_{d-a_k}\,.
\end{align*}
Hence we obtain Equation \eqref{eq:J-in-I-plus-C} in degree $d$, completing
the induction.

Let $N=\{0,1,\ldots,n\}\setminus \{k\}$. 
Let $N_{< t}:=\{j\in N\,|\, a_j< t\}$.
For $t\geqslant 0$, suppose we have 
defined $\psi_j$ for $j\in N_{<t}$. 
Set
\[
\mathcal I^{(t)} := \Big(\{f_l-\psi_lf_k\}_{l\in N_{<t}},\
\{f_l-\phi_lf_k\}_{l\in N\setminus N_{<t}}\Big),
\]
so $\mathcal I^{(0)} = \mathcal I$. We will 
inductively construct $\psi_j \in C_{a_j-a_k}$
and show $\mathcal I^{(t)} = \mathcal I$ 
for every $t\geqslant 0$.

\emph{Base case $t=0$.} $\mathcal I^{(0)} = \mathcal I$ by definition.

\emph{Inductive step.} Suppose $\mathcal I^{(t)} = \mathcal I$ 
for some $t \geqslant 0$. If $N_{<t}=N_{<t+1}$, 
then clearly $\mathcal I^{(t+1)} = \mathcal I$ and there is nothing 
to prove. So let us assume that $N_{<t}\subsetneqq N_{<t+1}$.
We construct $\psi_j \in C_{a_j-a_k}$ for $j\in N_{<t+1}\setminus N_{<t}$
and show that $\mathcal I^{(t+1)} = \mathcal I^{(t)} =\mathcal I$.

Let $j\in N_{<t+1}\setminus N_{<t}$.
Write $\phi_j = \alpha_j + \beta_j$ with $\alpha_j \in J_{a_j-a_k}$ and
$\beta_j \in C_{a_j-a_k}$. By \eqref{eq:J-in-I-plus-C} applied in degree
$a_j-a_k$,
\[
\alpha_j = u_j + f_k\sigma_j, \qquad \text{ {for some} } u_j \in \mathcal I_{a_j-a_k},\
\sigma_j \in C_{a_j-2a_k}.
\]
By the inductive hypothesis, $\mathcal I = \mathcal I^{(t)}$, so 
$u_j\in \mathcal I^{(t)}_{a_j-a_k}$. 
Moreover, since ${\rm deg}(u_j)<a_j$,
it is a $T$-combination of generators
whose degree is strictly less than $a_j$. Note that 
as $j\in N_{<t+1}\setminus N_{<t}$, we have $a_j=t$,
and so $u_j$ is a $T$-combination of generators
of $\mathcal I^{(t)}$
whose degree is strictly less than $t$, which is precisely 
\[
\{f_l-\psi_lf_k\}_{l\in N_t}\,.
\]
Define $\psi_j := \beta_j + f_k\sigma_j$. Since $f_k\sigma_j \in
f_kC_{a_j-2a_k} \subset C_{a_j-a_k}$ and $C_{a_j-a_k}$ is a subspace,
$\psi_j \in C_{a_j-a_k}$. Moreover
\[
\phi_j = \psi_j + u_j.
\]
Thus, we may consider the generator $f_j - \phi_jf_k$
of $\mathcal I^{(t)}$ and write 
\[
f_j - \phi_jf_k = (f_j-\psi_jf_k) - u_jf_k\,.
\]
Since $u_j\in (f_l-\psi_lf_k\,|\,{l\in N_{<t}})$, 
it follows 
$$u_jf_k\in (f_l-\psi_lf_k\,|\,{l\in N_{<t}})\subset \mathcal I^{(t)}\,.$$ 
Hence $f_j-\psi_jf_k\in \mathcal I^{(t)}$. It also shows 
that in $\mathcal I^{(t)}$ we may replace the generator $f_j-\phi_jf_k$
with $f_j-\psi_jf_k$. We may do this for every $j\in N_{<t+1}\setminus N_{<t}$.
Hence
\[
\Big(\{f_l-\psi_lf_k\}_{l\in N_{<t}},\
\{f_l-\phi_lf_k\}_{l\in N\setminus N_{<t}}\Big)
=
\Big(\{f_l-\psi_lf_k\}_{l\in N_{<t+1}},\
\{f_l-\phi_lf_k\}_{l\in N\setminus N_{<t+1}}\Big)\,.
\]
Thus, $\mathcal I^{(t)}=\mathcal I^{(t+1)}$.
By induction, $\mathcal I^{(t)} = \mathcal I$ for all $t\geqslant 0$. 
Taking $t\gg 0$ gives
\[
\mathcal I = \big(f_0-\psi_0f_k,\ldots,f_{k-1}-\psi_{k-1}f_k,\ f_{k+1}-\psi_{k+1}f_k,\ldots,f_n-\psi_nf_k\big),
\]
with $\psi_i \in C_{a_i-a_k}$ for each $i$, as required.
\end{proof}

We now prove the main result of this section.
\begin{theorem}\label{Thm:VSP-main}
Let $F^\perp = (f_0, f_1, \ldots,f_n) \subseteq T$ be a complete intersection.
Let $\deg f_i = a_i$ and assume that  $a_0 \leqslant a_1\leqslant \ldots \leqslant a_n$. 
Assume that $f_i$ is non-reduced for some $i$. Let $k$ be the smallest 
such that $f_k$ is non-reduced. 
Let $f_k=h^2g$ where ${\rm deg}(h)\geqslant 1$. Assume that the subscheme 
$V(f_0,f_1,\ldots,f_{k-1})\subset \mathbb P^n$
is reduced.  
Let $J$ be the complete-intersection $(f_0,\ldots,f_{k-1},f_{k+1},\ldots,f_n)$
and let $HF(T/J,-)$ be the Hilbert function of $T/J$. Then ${\rm VSP}(F)$
is irreducible of dimension 
$$\sum_{\substack{i=0\\i\neq k}}^n HF(T/J,a_i-a_k)\,.$$
\end{theorem}
\begin{proof}
We argue as in the proof of \cite[Proposition 28]{BBT13}.
Let 
$$\mathcal B':=\prod_{\substack{i=0\\i\neq k}}^n C_{a_i-a_k}\,.$$
Consider the open subset $U\subset \mathcal B'$ 
consisting of elements $(\psi_{0},\ldots, \psi_{k-1},\psi_{k+1},\ldots,\psi_n)$
such that the ideal 
$\mathcal I=(f_0-\psi_{0}f_k,\ldots,f_{k-1}-\psi_{k-1}f_k,f_{k+1}-\psi_{k+1}f_k,\ldots, f_n-\psi_nf_k)$
defines a reduced subscheme of $\mathbb P^n$. 
By Theorem \ref{vspthm-unified} and Lemma \ref{lem:existence-reduced-form}
it follows that $U$ maps surjectively onto $\mathcal V$.
By Lemma \ref{unique in C_i} it follows that this map $U\to \mathcal V$
is a bijection. Thus, ${\rm VSP}(F)$ is irreducible of 
dimension equal to ${\rm dim}(U)={\rm dim}(\mathcal B')$.
\end{proof}

The proposition below gives a sufficient 
condition for the highest degree generator of 
$F^{\perp}$ to be non-reduced.
    \begin{proposition}\label{prop a_n large}
        Let $(f_0,\ldots,f_n)\subset T$ be a complete intersection with $f_i$ 
        homogeneous polynomials. Let $a_i={\rm deg}(f_i)$. If 
        $a_n \geqslant \left(\sum\nolimits_{i=0}^{n-1}a_i\right)-n+2$,
        then we can find a linear form $\ell$ and a form $g$ 
        of degree $a_n-2$ such that $(f_0,\ldots,f_n)=(f_0,\ldots,f_{n-1},\ell^2g)$.
    \end{proposition}
    \begin{proof}
        Let $Z$ denote the subscheme of $\mathbb P^n$ defined by the 
        homogeneous ideal $(f_0,\ldots,f_{n-1})$. Then $Z$ is zero-dimensional.
        Consider the following exact sequence of sheaves on $\mathbb P^n$, which 
        is part of the Koszul resolution of $\mathcal O_{Z}$ (after twisting by $b$):
        $$\bigoplus_{i=0}^{n-1}\mathcal O_{\mathbb P^n}(b-a_i)\to 
            \mathcal O_{\mathbb P^n}(b)\to \mathcal O_{Z}(b)\to 0\,$$
        corresponding to the regular sequence $(f_0,\ldots,f_{n-1})$.
        Let $\tilde I$ denote the ideal sheaf of the scheme $Z$. Then $\tilde I(b)$
        is the image of the leftmost sheaf in the above sequence. 

        If $0\to \mathcal F' \to \mathcal F \to \mathcal F''\to 0$
        is a short exact sequence of sheaves, $\mathcal F'$ is $(m+1)$-regular
        and $\mathcal F$ is $m$-regular, then $\mathcal F''$ is $m$-regular. 
        The sheaf $\mathcal O(-m)$ is $m$-regular. Using the
        Koszul resolution of $\tilde I$, it easily follows that $\tilde I$
        is $\left(\left(\sum\nolimits_{i=0}^{n-1}a_i\right)-n+1\right)$-regular.
        It follows that for all $b\geqslant \left(\sum\nolimits_{i=0}^{n-1}a_i\right)-n$
        we have an exact sequence of vector spaces  
        \begin{equation}\label{e1:prop a_n large}
        H^0\left(\mathbb P^n,\bigoplus_{i=0}^{n-1}\mathcal O_{\mathbb P^n}(b-a_i)\right)\to H^0(\mathbb P^n,\mathcal O_{\mathbb P^n}(b))\to 
            H^0(\mathbb P^n,\mathcal O_{Z}(b))\to 0\,.    
        \end{equation}
        Let $\ell$ denote a general linear form which does not vanish
        at any point in the support of $Z$. Then 
        $$H^0(\mathbb P^n,\mathcal O_Z(b-2)) \xrightarrow{\ell^2} H^0(\mathbb P^n,\mathcal O_Z(b))$$
        is an isomorphism. Assume $b\geqslant \left(\sum\nolimits_{i=0}^{n-1}a_i\right)-n+2$.
        Then from the above discussion, we have the following commutative square
        in which the vertical arrows are surjections:
        \begin{center}
        \begin{tikzcd}
            H^0(\mathbb P^n, \mathcal O_{\mathbb P^n}(b-2)) \arrow[r,"\ell^2"]\arrow[d, two heads]& H^0(\mathbb P^n, \mathcal O_{\mathbb P^n}(b))\phantom{\,.}\arrow[d,two heads]\\
            H^0(\mathbb P^n, \mathcal O_Z(b-2)) \arrow[r,"\ell^2","\sim"']&H^0(\mathbb P^n, \mathcal O_Z(b))\,.
        \end{tikzcd}
        \end{center}
        Thus, if 
        $b\geqslant \left(\sum\nolimits_{i=0}^{n-1}a_i\right)-n+2$ then the 
        composite map 
        $$H^0(\mathbb P^n,\mathcal O_{\mathbb P^n}(b-2))\xrightarrow{\cdot \ell^2}
            H^0(\mathbb P^n,\mathcal O_{\mathbb P^n}(b))\to 
            H^0(\mathbb P^n,\mathcal O_{Z}(b))$$
        is surjective. 
        Let $f_n$ be any form of degree $a_n\geqslant \left(\sum\nolimits_{i=0}^{n-1}a_i\right)-n+2$. 
        It follows that there is a form $g$ of degree $a_n-2$
        such that $f_n-\ell^2g$ is in the kernel of the map 
        $$H^0(\mathbb P^n,\mathcal O_{\mathbb P^n}(b))\to 
            H^0(\mathbb P^n,\mathcal O_{Z}(b))\,.$$
        From \eqref{e1:prop a_n large} it follows that 
        $f_n-\ell^2g\in (f_0,f_1,\ldots,f_{n-1})$. Thus, 
        $(f_0,\ldots,f_{n})=(f_0,\ldots,f_{n-1},\ell^2g)$.
    \end{proof}

 In Proposition \ref{prop a_n large} if $V(f_0, \ldots, f_{n-1})$ 
 is reduced, we obtain the following result as a consequence 
 of Theorem \ref{Thm:VSP-main}.
\begin{corollary}\label{show fn is non-reduced}
Let $F^\perp=(f_0,f_1,\ldots,f_n)$ be a complete intersection. 
Let $a_i={\rm deg}(f_i)$ and assume that $a_0\leqslant a_1\leqslant \ldots\leqslant a_n$. 
Further assume that $a_n\geqslant \left(\sum\nolimits_{i=0}^{n-1}a_i\right)-n+2$
and that the subscheme $V(f_0,f_1,\ldots,f_{n-1})\subset \mathbb P^n$
is reduced. 
Let $J$ be the complete-intersection $(f_0,\ldots,f_{n-1})$
and let $HF(T/J,-)$ be the Hilbert function of $T/J$. Then ${\rm VSP}(F)$
is irreducible of dimension 
$$\sum_{\substack{i=0}}^{n-1} HF(T/J,a_i-a_n)\,.$$    
\end{corollary}
\begin{proof}
    Using Proposition \ref{prop a_n large} it follows that 
    $f_n=\ell^2g$ for some linear form $\ell$. As the subscheme 
    $V(f_0,f_1,\ldots,f_{n-1})\subset \mathbb P^n$ is reduced,
    we obtain that all the $f_i$ are reduced for $0\leqslant i\leqslant n-1$.
    Now we apply Theorem \ref{Thm:VSP-main} to obtain the result. 
\end{proof}

\begin{proof}[Proofs of Theorem \ref{theorem-intro1} and Theorem \ref{theorem-intro2}]
    Theorem \ref{theorem-intro1} follows from Theorem \ref{vspthm-unified}
    and Theorem \ref{Thm:VSP-main}. Theorem \ref{theorem-intro2}
    follows from Theorem \ref{vspthm-unified} and Corollary \ref{show fn is non-reduced}.
\end{proof}

\section{Final remarks}
If $F$ is a homogeneous form such that $F^{\perp}$ is a 
complete intersection, then it may not be true that every 
minimal apolar ideal is a complete intersection ideal. In fact, 
the following example shows that there exists a form $F$ 
such that  $F^{\perp}$ is a complete intersection but no 
minimal apolar ideal is a complete intersection.

\begin{example}\label{counterexample}
Let $I= (Y^2-Z^2, XY, X^5-Z^5) \subset T=\mathbb{C}[X,Y,Z]$. 
Note that $I$ is an Artinian complete intersection ideal. By Macaulay2 \cite{M2}
we obtain that for $F = (x^5+y^5)z+y^3z^3+yz^5$, we have $F^{\perp} = I$.
We claim that any minimal 
apolar ideal of $F$ is not a complete intersection.
\begin{proof}
In \cite[Example 6.1]{BCMS26} it is proved that $4 < \rk (F) \leqslant 8$.
We now show that any minimal apolar 
ideal of $F$ is not a complete intersection. 
We prove this by contradiction. Suppose there 
exists a minimal apolar ideal 
$I(\mathbb{X}) \subset F^{\perp}$ such that 
$I(\mathbb{X}) = (g_1, g_2)$ is a complete 
intersection ideal. Since $I(\mathbb X)\subset F^\perp$,
$\deg (g_i)\geqslant 2$.
Without loss of generality assume 
$\deg (g_1) \leqslant \deg (g_2)$.
Since
$\deg (g_1) \cdot \deg (g_2) = |\mathbb{X}|=\rk(F)$, 
possible degrees for $g_1, g_2$ are $(2,3), (2,4)$. 
But any such ideal 
$I(\mathbb{X}) \subset (XY, Y^2-Z^2)$ since 
$(F^{\perp})_i =(XY, Y^2-Z^2)_i$ for 
$2 \leqslant i \leqslant 4$. 
This implies $V(XY, Y^2-Z^2)\subset \mathbb X$ 
which is a contradiction as $\mathbb X$ is a reduced zero
dimensional scheme whereas $V(XY, Y^2-Z^2)$
is not reduced. Hence, in this case, no minimal 
apolar ideal of $F$ is a complete intersection.
\end{proof}
\end{example}

\bibliography{main}
\bibliographystyle{halpha}

\end{document}